\documentclass[a4paper,12pt]{amsart}

\usepackage{amssymb,amsmath}
\usepackage{amsthm}

\newtheorem{defin}{Definition}[section]
\newtheorem{thm}[defin]{Theorem}
\newtheorem{prop}[defin]{Proposition}
\newtheorem{lemma}[defin]{Lemma}

\newtheorem{rem}[defin]{Remark}

\numberwithin{equation}{section}

\DeclareMathOperator{\supp}{\mathrm{supp}}
\DeclareMathOperator{\vers}{\mathrm{vers}}

\def\R {\mathbb{R}}       
\def\N {\mathbb{N}}       
\def\Z {\mathbb{Z}}
\def\Rn {\mathbb{R}^{n}}

\def\cC{{\mathcal C}}

\def\cF{{\mathcal T}}

\def\cN{{\mathcal N}}

\def\cR{{\mathcal R}}
\def\cS{{\mathcal S}}

\def\cV{{\mathcal V}}

\def\1{\lambda}
\def\2{\Sigma_{2}}

\def\<{{\langle}}
\def\>{{\rangle}}
\def\norm#1{{\<} #1 {\>}}

\def\S{{\Sigma}}

\def\va{\varphi}
\def\Rn{{\R^n}}
\def\jp#1{{\langle} #1 {\rangle}}

\newcommand{\tp}{\widetilde{\psi}}

\title{Global $L^p$ continuity of Fourier integral operators}

\author[Sandro Coriasco]{Sandro Coriasco}
\address{
  Sandro Coriasco:
  \endgraf
  Department of Mathematics
  \endgraf
  University of Torino
  \endgraf
   V. C. Alberto, n. 10 
  \endgraf
   Torino I-10126
  \endgraf
  Italy
  \endgraf
  {\it E-mail address} {\rm sandro.coriasco@unito.it}
  }  

\author[Michael Ruzhansky]{Michael Ruzhansky}
\address{
  Michael Ruzhansky:
  \endgraf
  Department of Mathematics
  \endgraf
  Imperial College London
  \endgraf
  180 Queen's Gate, London SW7 2AZ 
  \endgraf
  United Kingdom
  \endgraf
  {\it E-mail address} {\rm m.ruzhansky@imperial.ac.uk}
  }

\thanks{The second
 author was supported in part by the EPSRC
 grants EP/E062873/1 and EP/G007233/1.}
\date{\today}

\subjclass{Primary 35S30; Secondary 42B30, 46E30, 47B34}
\keywords{Fourier integral operators, global 
$L^p(\Rn)$ boundedness}

\begin{document}


\begin{abstract}
In this paper we establish global $L^p$ regularity properties
of Fourier integral operators. The orders of decay of the
amplitude are determined for operators to be bounded
on $L^p(\Rn)$, $1<p<\infty$, as well as to be bounded 
from Hardy space $H^1(\Rn)$ to $L^1(\Rn)$. The obtained
results extend local $L^p$ regularity
properties of Fourier integral operators
established by Seeger, Sogge and Stein (1991) as
well as global $L^2(\Rn)$
results of Asada and Fujiwara (1978) and
Ruzhansky and Sugimoto (2006), to the global
setting of $L^p(\Rn)$. Global boundedness
in weighted Sobolev spaces $W^{\sigma,p}_s(\Rn)$ is
also established. The techniques used in the proofs are 
the space dependent dyadic decomposition and 
the global calculi developed by Ruzhansky and Sugimoto (2006) 
and Coriasco (1999).
\end{abstract}

\maketitle

\section{Introduction}
\label{sec:intro}

In this paper we investigate global $L^p(\Rn)$ continuity
properties of non-degenerate
Fourier integral operators. In particular, we are
interested in the question of what decay properties of
the amplitude guarantee the global boundedness of
Fourier integral operators from $L^p(\Rn)$ to $L^p(\Rn)$.

The analysis of the local $L^2$ boundedness of 
Fourier integral operators goes back to
Eskin \cite{Es70} and H\"ormander \cite{Ho71}, who showed
that non-degenerate Fourier integral operators with
amplitudes in the symbol class $S^0_{1,0}$ are
locally bounded on $L^2(\Rn)$. A Fourier integral operator
of class $I^\mu(X,Y;{\cC})$  is called non-degenerate 
if its canonical relation $\cC$ is locally a graph 
of a symplectic mapping from 
$T^*X\backslash 0$ to $T^*Y\backslash 0$.
If the canonical relation of
the operator degenerates, the local $L^2$ boundedness of
zero order operators is known to fail, see e.g.
H\"ormander \cite{Ho85}. In this paper we will be concerned
with non-degenerate operators only.

Since '70s this local $L^2$ boundedness result has been 
extended in different directions. On one hand, 
global $L^2(\Rn)$ boundedness has been studied,
motivated by applications in microlocal analysis and
hyperbolic partial differential equations. On the other
hand, its extension to $L^p$ spaces with
$p\not=2$ has been also under
study motivated by applications in harmonic analysis.

The question of the global $L^2(\Rn)$ boundedness
has been first widely investigated in the case of pseudo-differential operators.
The phase is trivial in this case, so the main question 
is to determine minimal assumptions on the amplitude
which guarantees the global $L^2(\Rn)$ boundedness.
For example, one wants to relax an assumption that
the symbol of a pseudo-differential operator is in the
symbol class $S^0_{0,0}$ for operators to be still
bounded on $L^2(\Rn)$.
There are different sets of assumptions, see e.g.
Calder\'on and Vaillancourt \cite{CV71},
Childs \cite{Ch76}, Coifman and Meyer \cite{CM78},
Cordes \cite{Co75}, Sugimoto \cite{Su88}, etc.
The question of global $L^2(\Rn)$ boundedness of 
Fourier integral operators is more subtle, and involves
different sets of assumptions on both phase and 
amplitude. Operators arising in applications to
hyperbolic equations and Feynman path integrals
have been considered e.g. in 
Asada \cite{As81}, Asada and Fujiwara \cite{AF78},
Kumano-go \cite{Ku76}, Boulkhemair \cite{Bo97}.
On the other hand, applications to smoothing estimates
for evolution partial differential equations require
less restrictive assumptions on the phase, and the
necessary estimates have been established by
Ruzhansky and Sugimoto \cite{RS06a,RS06}.

Local $L^p$ boundedness of
Fourier integral operators has been under intensive
study as well. In the case of $p\not=2$ there is a
loss of derivatives in $L^p$-spaces. For example,
a loss of $(n-1)|1/p-1/2|$ derivatives
has been established for operators appearing as
solutions to the wave equations, see e.g.
Beals \cite{Be82}, Peral \cite{Pe80},
Miyachi \cite{Mi80}. Finally, Seeger, Sogge and 
Stein \cite{SSS91} showed that general non-degenerate
Fourier integral operators in the class
$I^\mu(\Rn,\Rn;\cC)$ are locally bounded in $L^p(\Rn)$
provided that their amplitudes are in the
class $S^\mu_{1,0}$ with $\mu\leq -(n-1)|1/p-1/2|$,
$1<p<\infty$ (see also
Sogge \cite{So93} and Stein \cite{St93}). 
In the case of $p=1$, they showed that
operators of order $\mu=-(n-1)/2$ are locally
bounded from the Hardy space $H^1$ to $L^1$,
while Tao \cite{Ta04} showed that operators
of the same order are also locally of weak type (1,1).
Extensions of these results with smaller loss of
regularity under additional geometric
assumptions on the canonical relations have been
studied by Ruzhansky \cite{Ru00,Ru01}. 

The aim of this paper is to establish global $L^p(\Rn)$
boundedness of Fourier integral operators, which
depends on the growth/decay order of the
amplitude in $x$ and $y$ variables. The results of
this paper will extend the local $L^p$ results of
Seeger, Sogge and Stein \cite{SSS91} as well as
global $L^2$ results of Asada and Fujiwara \cite{AF78}, Coriasco \cite{Co99},
and Ruzhansky and Sugimoto \cite{RS06}, to the global
setting of $L^p(\Rn)$. In fact, for $p\not=2$,
we will observe that there is a loss not only of
derivatives but also of growth/decay dependent 
on the value of $p$. Both of these
losses disappear in the case $p=2$. Consequently,
using the global calculi of Fourier integral operators
developed by Coriasco \cite{Co99} and by Ruzhansky and Sugimoto 
\cite{RS06b,RS07}, we can also obtain global weighted
estimates in Sobolev spaces $W^{s,p}_\sigma(\Rn)$.

We will be initially concerned with operators $\cF$ of the form
\begin{equation}\label{eq:FIO-form}
         (\cF u)(x)=\int_\Rn \int_\Rn 
        e^{i[\langle x,\xi \rangle - \varphi(y,\xi)]} 
        b(x,y,\xi) u(y) 
                         \, dy d\xi,
\end{equation} 
where $\va$ is a real-valued phase function, positively
homogeneous of order one in $\xi$, and 
$b$ is an amplitude. Local $L^p$ properties of such operators
were considered by Seeger, Sogge and Stein \cite{SSS91}
and their global $L^2$ properties were analysed by
Ruzhansky and Sugimoto \cite{RS06a}. We note that 
a general H\"ormander's
Fourier integral operator can be always written
in the form \eqref{eq:FIO-form} microlocally while there are
in general topological obstructions globally. 
The microlocal qualitative properties of such operators
are well-known, see e.g. H\"ormander \cite{Ho71,Ho85} or
Duistermaat \cite{Du96}.
Since the aim of
this paper is to investigate $L^p$ properties rather than
trivialisations of Maslov index, we will treat operators that can be written in the form 
\eqref{eq:FIO-form}
globally. We note that operators \eqref{eq:FIO-form} 
and their adjoints appear as
propagators to hyperbolic partial differential equations
as well as canonical transforms in smoothing problems.

Subsequently, we will deal with Fourier integral operators of the form
\begin{equation}\label{eq:SGFIO}
        Au(x)=\int_\Rn 
        e^{i\varphi(x,\xi)} a(x,\xi) \widehat{u}(\xi)d\xi,
\end{equation}
where $\varphi$ is as above and the amplitude $a$ does not depend on $y$.

Finally, we mention that 
results on the local $L^p$ boundedness of 
Fourier integral operators with complex valued phase
functions have been established by Ruzhansky \cite{Ru01}, extending previous
local $L^2$ results by Melin and Sj\"ostrand \cite{MS75}
and H\"ormander \cite{Ho83}, and that there are also results 
in $(\mathcal F L^p)_{comp}$ spaces and in modulation spaces
by Cordero, Nicola and Rodino \cite{CNR07}.

Constants in this paper will be denoted by letters $C$
and their values may vary even in the same formula.
If the value of a constant is important and unchanged
in a calculation, we will use sub-indices, denoting it
e.g. by $C_1$, $C_2$, etc. We will denote
$\jp{x}=(1+|x|^2)^{1/2}$. 
Occasionally, for functions $f(x,y,\xi,w),g(x,y,\xi,w)$, 
$x,y,\xi\in\R^n$, and $w$ varying in a suitable parameter 
space, we will write
$f\prec g$, $f\succ g$, if there exist constants 
$A,B>0$ independent of $w$ such that, for arbitrary $x,y,\xi,w$,
we have
$|f(x,y,\xi,w)| \le A|g(x,y,\xi,w)|$, 
$|f(x,y,\xi,w)|\ge B|g(x,y,\xi,w)|$, respectively.
If both $f\prec g$ and $f\succ g$ hold, we will write
$f\sim g$. By $B_R(y)$ we will denote an open ball with
radius $R$ centred at $y$.

\section{Main results}
\label{sec:lp}

Let operator $\cF$ be given by 
\begin{equation}
        \label{eq:FIO}
        (\cF u)(x)=\int_\Rn \int_\Rn 
        e^{i[\langle x,\xi \rangle - \varphi(y,\xi)]} 
        b(x,y,\xi) u(y) 
                         \, dy d\xi,
\end{equation}
with a real-valued phase $\va$ and amplitude $b$.  
The main result of this paper is the following
\begin{thm}
        \label{thm:main}
        Let $1 < p <\infty$ and $m,\mu\in\R$. 
        Let $\cF$ be operator \eqref{eq:FIO}, where
        $\varphi\in \cC^\infty(\R^n\times(\R^n\setminus\{0\}))$
        is real-valued and
        positively homogeneous of order $1$ in $\xi$, 
        i.e. that $\va(y,\tau\xi)=\tau\va(y,\xi)$ 
        for all $\tau>0$ and $\xi\not=0$.
  Assume that $\xi\not=0$ on $\supp b$ and assume
  one of the following properties:      
\begin{itemize}
\item[(I)] Let $\varphi$ be such that  
for all $x\in\Rn$ and $\xi\in\Rn\backslash 0$ we have
\begin{equation}
        \label{eq:phf}
\begin{aligned}
        |\det\partial_y\partial_\xi\va(y,\xi)|\geq C>0,
        \; \partial^\alpha_y\varphi(y,\xi) \prec 
        \norm{y}^{1-|\alpha|} |\xi|
        \textrm{ for all } \alpha,\\ 
        \;
        \norm{\nabla_\xi\varphi(y,\xi)}\sim\norm{y},
        \;
        \norm{d_y\varphi(y,\xi)}\sim\norm{\xi},
\end{aligned}
\end{equation}
and such that 
\begin{equation}\label{EQ:phf-add}
\partial^\alpha_x\partial^\beta_\xi\va(y,\xi)\prec 1
\end{equation} 
for all multi-indices $\alpha,\beta$
such that $|\alpha+\beta|\ge2$.\\
Let $b\in\cC^\infty(\R^n\times\R^n\times\R^n)$ satisfy
\begin{equation}
        \label{eq:amplf}
        \partial_x^\alpha\partial_y^\beta\partial_\xi^\gamma 
        b(x,y,\xi)
        \prec \norm{x}^{m_1}\norm{y}^{m_2}\norm{\xi}^{\mu-|\gamma|}
\end{equation}
for all $x,y,\xi\in\R^n$ and all multi-indices 
$\alpha,\beta,\gamma$, with some $m_1, m_2\in\R$
such that $m_1+m_2=m$.

\item[(II)] Let $\varphi$ satisfy \eqref{eq:phf} 
on $\supp b$, and 
 \begin{equation}
        \label{eq:thm-phi1}
       \partial^\alpha_y\partial^\beta_\xi\va(y,\xi)\prec 1
\end{equation}
for all $x,y,\xi$ on $\supp b$ and all $\alpha, \beta$
such that $|\alpha|\geq 1$ and $|\beta|\geq 1$, and let
$b\in\cC^\infty(\R^n\times\R^n\times\R^n)$ satisfy
\begin{equation}
        \label{eq:amplf-decay1}
        \partial_x^\alpha\partial_y^\beta\partial_\xi^\gamma 
        b(x,y,\xi)
        \prec \norm{x}^{m_1-|\alpha|}
        \norm{y}^{m_2}\norm{\xi}^{\mu-|\gamma|}
\end{equation}
for all $x,y,\xi\in\R^n$ and all multi-indices 
$\alpha,\beta,\gamma$, with some $m_1, m_2\in\R$
such that $m_1+m_2=m$.
\item[(III)] Let $\varphi$ satisfy \eqref{eq:phf} 
on $\supp b$, and 
 \begin{equation}
        \label{eq:thm-phi2}
    \partial^\alpha_y\partial^\beta_\xi\va(y,\xi)\prec 
    \jp{y}^{1-|\alpha|}
\end{equation}
for all $x,y,\xi$ on $\supp b$ and all $\alpha, \beta$
such that $|\beta|\geq 1$, and let
$b\in\cC^\infty(\R^n\times\R^n\times\R^n)$ satisfy
\begin{equation}
        \label{eq:amplf-decay2}
        \partial_x^\alpha\partial_y^\beta\partial_\xi^\gamma 
        b(x,y,\xi)
        \prec \norm{x}^{m_1}
        \norm{y}^{m_2-|\beta|}\norm{\xi}^{\mu-|\gamma|}
\end{equation}
for all $x,y,\xi\in\R^n$ and all multi-indices 
$\alpha,\beta,\gamma$, with some $m_1, m_2\in\R$
such that $m_1+m_2=m$.
\end{itemize}
        Then, $\cF$ extends to a
        bounded operator from $L^p(\R^n)$ to itself,
      provided that 
    \begin{equation}\label{EQ:m-mu}
     m\le-n\left|\frac{1}{p}-\frac12\right| 
        \textrm{ and }
         \mu\le-(n-1)\left|\frac{1}{p}-\frac12\right|.
    \end{equation}     
\end{thm}
Let us now discuss the assumptions of Theorem \ref{thm:main}.
First of all, 
we note that assumptions \eqref{eq:phf} are very natural
in the sense that they ask that $\va$ is essentially of
order one in both $y$ and $\xi$.
Condition
\begin{equation}\label{EQ:nondeg-ph}
 |\det\partial_y\partial_\xi\va(y,\xi)|\geq C>0,
\end{equation} 
for all $y\in\Rn$ and $\xi\in \Rn\backslash 0$
is simply a global version of the local
graph condition of the non-degeneracy of Fourier integral 
operator \eqref{eq:FIO}.
Assumption \eqref{eq:amplf} says that $b$ has a symbolic
behaviour in $\xi$ and is of order $m_1+m_2=m$
jointly in $x$ and $y$. 

We assume that $\xi\not=0$
on the support of $b$ to avoid the singularity of the 
phase at the origin. We note that this issue does not
arise in local boundedness problems (as in \cite{SSS91})
since the corresponding part of the operator is locally
smoothing. In our situation it is still smoothing but
may destroy the behaviour with respect to $x$ and $y$.
Some global results in $L^2(\Rn)$ for small frequencies
have been established by Ruzhansky and Sugimoto in
\cite{RS06a} using weighted estimates for multipliers
of Kurtz and Wheeden \cite{KW}, and we refer to
\cite{RS06a} for a discussion of complications
that arise in this situation.

Assumption (II) is different from (I) in that
we do not assume the boundedness 
\eqref{EQ:phf-add}, and assume boundedness only of mixed
derivatives (i.e. $|\alpha|\geq 1$ and $|\beta|\geq 1$),
but in addition assume that
derivatives of $b$ have some decay properties 
in \eqref{eq:amplf-decay1} or in
\eqref{eq:amplf-decay2}.
In assumption (III) we also allow non-mixed derivatives
(i.e. $\partial_\xi^\beta$-derivatives
when $\alpha=0$) to grow in $y$.
Moreover, in both (II) and (III) we assume 
\eqref{eq:phf} to hold only on the support of $b$.

We note that propagators for hyperbolic 
partial differential equations lead to operators
\eqref{eq:FIO} with $b(x,y,\xi)=b(y,\xi)$ independent of
$x$, in which case assumption 
\eqref{eq:amplf-decay1} becomes trivial if $\alpha\not=0$. 
For these
propagators also the boundedness \eqref{EQ:phf-add} is
satisfied under natural assumptions on the symbol of
the hyperbolic equation. However,
we do not always want to assume the boundedness
\eqref{EQ:phf-add} since it fails for
non-mixed derivatives
(i.e. when $\alpha=0$ or $\beta=0$), e.g. in applications
to smoothing estimates for dispersive equations.
For example, it is shown in \cite{RS06a,RS06} that
for canonical transforms appearing there
condition \eqref{EQ:phf-add} fails, but it is also shown
that additional decay of derivatives as in
\eqref{eq:amplf-decay1} or \eqref{eq:amplf-decay2} holds.

If the amplitude $b$ in Theorem \ref{thm:main} is compactly
supported in $(x,y)$, Theorem \ref{thm:main} implies the
local $L^p$ boundedness under the assumptions in
Seeger, Sogge and Stein \cite{SSS91}, implying,
in particular, that the order $\mu$ in Theorem
\ref{thm:main} cannot be improved in general.
Let us now give some explanation about the order $m$.
In \cite{CNR07}, Cordero, Nicola and Rodino investigated
the question of boundedness of Fourier integral operators
on $(\mathcal F L^p(\Rn))_{comp}$, the space of compactly supported
distributions where Fourier transform is in $L^p(\Rn)$. 
They proved that if the amplitude of an operator is
of order $-n\left|\frac{1}{p}-\frac12\right|$ 
in $\xi$ (plus additional assumptions), then the operator is 
continuous on $(\mathcal F L^p(\Rn))_{comp}$. They also showed that
this order of decay is sharp by constructing a 
counterexample for higher orders.
Roughly
speaking, the conjugation with the Fourier transform
interchanges the roles of $x$ and $\xi$, so the orders in
\cite{CNR07} correspond to orders 
$m=-n\left|\frac{1}{p}-\frac12\right|$ 
and $\mu=-\infty$ for operators
in the setting of Theorem \ref{thm:main} since the assumption
of the compact support in $(\mathcal F L^p(\Rn))_{comp}$
corresponds to locally smoothing operators in $L^p(\Rn)$.
From this point of view, Theorem \ref{thm:main} also improves
the result of \cite{CNR07} with respect to $\mu$ to the
order $\mu=-(n-1)\left|\frac{1}{p}-\frac12\right|$, 
which cannot be
improved further in general. However, the order $m$ in Theorem
\ref{thm:main} can still be improved if we restrict 
the size of the support while still allowing it to
move to infinity. In this case a uniform estimate is
possible for 
$m\leq -(n-1)\left|\frac{1}{p}-\frac12\right|$ 
and it is given in Theorem \ref{COR:small-Q}.
The same improved threshold for the order $m$ 
can be achieved for the Fourier integral operators \eqref{eq:SGFIO}
considered by Coriasco \cite{Co99}, as stated in Theorem \ref{thm:main2}.

To prove Theorem \ref{thm:main} 
we use interpolation between the $L^2(\Rn)$-boundedness
and boundedness from the Hardy space $H^1(\Rn)$ to
$L^1(\Rn)$. The global $L^2(\Rn)$-bound\-ed\-ness under assumptions
(I) and (II)--(III) would follow from the results of
Asada and Fujiwara \cite{AF78} and
Ruzhansky and Sugimoto \cite{RS06a}, respectively.
Thus, the main point is to prove the boundedness
from the Hardy space $H^1(\Rn)$ to $L^1(\Rn)$. 
This can be achieved by using the
atomic decomposition of $H^1(\Rn)$ and splitting the
argument for atoms with large and small supports.
However, there is a number of difficulties in this
argument compared with that of
\cite{SSS91}. For example, supports are 
no longer bounded and can become very large,
and hence, while this case is simple for the local boundedness,
it requires to be analysed further in the global setting. 
Another  global feature is
that even if the supports of atoms may be small,
they may still move to infinity (while
remaining small). We deal with this situation by 
introducing a dyadic decomposition in frequency which
depends on $y$. The dyadic pieces that we work with
are of the size
$2^{-k}$ in the radial direction and of the size 
$2^{-\frac{k}{2}}\jp{y}^{\frac12}$ in other directions
(tangential to the sphere in the frequency space).
Thus, we obtain the following theorem in the setting of
Hardy space $H^1(\Rn)$:
\begin{thm}
        \label{thm:submain}
        Let $\cF$ be the Fourier 
        integral operator \eqref{eq:FIO}. 
        Under the hypotheses of Theorem 
        \ref{thm:main}, 
        operator $\cF$ extends to a 
        bounded operator from the
        Hardy space $H^1(\R^n)$ to $L^1(\R^n)$,
      provided that 
      $m\leq -n/2$ and $\mu\leq -(n-1)/2$.
\end{thm}

\medskip
We can establish also a result in weighted
Sobolev spaces.
Let $W^{\sigma,p}_s(\Rn)$ denote the weighted Sobolev space,
i.e. the space of all $f\in\cS'(\Rn)$ such that
$\jp{x}^s (1-\Delta)^{\sigma/2} f(x)$ belongs to $L^p(\Rn)$.
\begin{thm}
        \label{thm:submain-Sobolev}
    Let $1<p<\infty$ 
        and let $\sigma,s\in\R$.     
        Let $\cF$ be the Fourier 
        integral operator \eqref{eq:FIO} as in
        Theorem \ref{thm:main} with
        orders $m,\mu\in\R$, and let
        $m_p= -n\left|\frac{1}{p}-\frac{1}{2}\right|$, 
        $\mu_p=-(n-1)\left|\frac{1}{p}-\frac{1}{2}\right|$. 
        Then operator $\cF$ extends to a 
        bounded operator from $W^{\sigma,p}_s(\R^n)$ to 
        $W^{\sigma-\mu-\mu_p,p}_{s-m-m_p}(\R^n)$.
\end{thm}
Theorem \ref{thm:submain-Sobolev} follows from
Theorem \ref{thm:main} and composition formulae
of Fourier integral operators with pseudo-differential 
operators as in \cite{RS06b} or in \cite{RS07}.
In fact, here we only need a special class of
pseudo-differential operators, namely of operators
with symbols $\pi_{s,\sigma}(x,\xi)=\jp{x}^s\jp{\xi}^\sigma$ for
which we have $(\textrm{Op }\pi_{s,\sigma}) 
(W^{\sigma,p}_s(\Rn))=L^p(\Rn)$.
Global composition formulae
of \cite{RS06b,RS07} will be also used in the proof of
Theorem \ref{thm:submain}.

The assumptions on the
order of the amplitude in Theorem \ref{thm:main}
can be relaxed
if we work with functions with
compact support. We will assume that the supports are
uniformly bounded but will still allow them to move to
infinity (while remaining bounded). In this 
situation the proof of 
Theorem \ref{thm:main} will also imply the following
\begin{thm}\label{COR:small-Q}
Let $1<p<\infty$ and let $m,\mu\in\R$.   
        Let $\cF$ be the Fourier 
        integral operator \eqref{eq:FIO} as in
        Theorem \ref{thm:main}.
        Let $R>0$.
 Let $\cV(\Rn)\subset L^p(\Rn)$ be a set of all functions
 $f\in L^p(\Rn)$ such that for every $f\in\cV(\Rn)$
 there exists $y\in\Rn$ such that
 $\supp f\subset B_R(y)$, and let $\cV(\Rn)$ have the topology
 induced by $L^p(\Rn)$. 
        Then operator $\cF$ extends to a 
        continuous operator from $\cV(\R^n)$ to 
        $L^p(\R^n)$, provided that
    \begin{equation}\label{EQ:m-mu2}
     m\le-(n-1)\left|\frac{1}{p}-\frac12\right| 
        \textrm{ and }
         \mu\le-(n-1)\left|\frac{1}{p}-\frac12\right|.
    \end{equation}    
\end{thm}
Theorem \ref{COR:small-Q} will follow from Remarks
\ref{REM:large-Q} and \ref{REM:small-Q}. 
We also have natural counterparts of Theorem \ref{COR:small-Q}
for $H^1$ and $W^{\sigma,p}_{s}$ as in Theorems
\ref{thm:submain} and \ref{thm:submain-Sobolev}.

Finally, by an argument similar to 
the one used in \cite{CNR08}, it is also possible to prove 
the $L^p$-continuity of the classes of 
Fourier integral operators considered in \cite{Co99}, 
where the phase function is assumed positively 
homogeneous of order $1$ in $\xi$ and satisfies 
\eqref{eq:phf}:

\begin{thm}\label{thm:main2}
Let $A=A_{\varphi,a}$ be a Fourier integral operator of the form
\begin{equation}\label{EQ:fio2}
        Au(x)=\int_\Rn 
        e^{i\varphi(x,\xi)} a(x,\xi) \widehat{u}(\xi)d\xi,
\end{equation}
with a real-valued phase function $\varphi$ such that 
$\va(y,\tau\xi)=\tau\va(y,\xi)$ 
        for all $\tau>0$ and $\xi\not=0$, 
        and assume that the condition \eqref{eq:phf} 
        holds true for all $x\in\Rn$ and $\xi\in\Rn\backslash 0$.
         Moreover, assume that $\xi\not=0$ on the 
         support of the amplitude $a$, 
         and that $a\in S^{m,\mu}$, i.e. that
\[
        \partial_x^\alpha\partial_\xi^\beta
        a(x,\xi)
        \prec \norm{x}^{m-|\alpha|}\norm{\xi}^{\mu-|\beta|},
\]
for all $x,\xi\in\R^n$ and all multi-indices 
$\alpha,\beta$, with some $m, \mu\in\R$. 
Then, $A$ extends to a bounded operator from $L^p(\R^n)$ 
to itself, provided that
    \begin{equation}\label{EQ:sharp-thresholds}
     m\le-(n-1)\left|\frac{1}{p}-\frac12\right| 
        \textrm{ and }
         \mu\le-(n-1)\left|\frac{1}{p}-\frac12\right|.
    \end{equation}
\end{thm}

The thresholds \eqref{EQ:sharp-thresholds} are sharp, 
by a modification of a counterexample described in 
\cite{CNR08}.  The improvement
in Theorem \ref{thm:main2} compared to that in
Theorem \ref{thm:main}, (III), comes from the independence
of the amplitude of $A$ on $y$-variable, if we write 
the adjoint $A^*$ in the form of an operator $\cF$ in
Theorem \ref{thm:main}.
The proof of Theorem 
\ref{thm:main2} is given in Section \ref{sec:sketch}. Finally, the composition formulae in \cite{Co99} together with Theorem \ref{thm:main2} imply the analog
of Theorem \ref{thm:submain-Sobolev} for the operator $A$:
\begin{thm}
        \label{thm:submain-Sobolev2}
    Let $1<p<\infty$ 
        and let $\sigma,s\in\R$.     
        Let $A$ be the Fourier 
        integral operator \eqref{eq:SGFIO} as in
        Theorem \ref{thm:main2} with
        orders $m,\mu\in\R$, and let
        $m_p= -(n-1)\left|\frac{1}{p}-\frac{1}{2}\right|$. 
        Then operator $A$ extends to a 
        bounded operator from $W^{\sigma,p}_s(\R^n)$ to 
        $W^{\sigma-\mu-m_p,p}_{s-m-m_p}(\R^n)$.
\end{thm}

\section{Proof of Theorem \ref{thm:submain}}
\label{sec:proof}

Since Theorem \ref{thm:main} follows by complex
interpolation from Theorem \ref{thm:submain} and
$L^2$-boundedness results in \cite{AF78} and
\cite{RS06a} under assumptions (I) and (II)--(III), respectively,
we need to prove Theorem \ref{thm:submain}.
This will be achieved through various subsequent steps.\\

Given $f\in H^1(\R^n)$, we can decompose
(see e.g. \cite{St93}) function
$\displaystyle f=\sum_Q\lambda_Q a_Q$, where 
$\sum_Q |\lambda_Q| \simeq \|f\|_{H^1(\R^n)}$ and the atoms 
$a_Q\in H^1(\R^n)$ have the following properties:
\begin{enumerate}
        \item $\supp a_Q \subset Q$, where 
        $Q\subset\R^n$ is a cube of sidelength $q$;
        \item $\|a_Q\|_{L^\infty(\R^n)}\le |Q|^{-1}$;
        \item $\displaystyle\int_Q a_Q(y) \, dy = 0$.
\end{enumerate}

Theorem \ref{thm:submain} would then follow
if we show that 
\begin{equation}
        \label{eq:estim}
        \|\cF a_Q\|_{L^1(\R^n)}\le C,
\end{equation} 
for a constant $C$ independent of $a_Q$.

Let $F=F(x,y)$ denote the distribution kernel 
of $\cF$, given by the oscillatory integral
\begin{equation}
        \label{eq:F}
        F(x,y)=\int_\Rn e^{i[\langle x,\xi\rangle 
        - \varphi(y,\xi)]} b(x,y,\xi) \, d\xi.
\end{equation}

We begin showing that the amplitude function can 
be assumed supported only in a suitable neighbourhood of 
the wave front set of the distributional kernel of $\cF$:
\begin{prop}
        Let $\chi=\chi(x,y,\xi)$ be supported in
        $E_k=\{(x,y,\xi)\in\R^n\times\R^n\times\R^n \,\colon\,
        |x-\nabla_\xi\varphi(y,\xi)| \le k \norm{x} \}$, 
        $k\in(0,1)$ suitably 
        small, and such that $\left.\chi\right|_{E_\frac{1}{2}}
        \equiv 1$.
        Moreover\footnote{
                With $h\in\cC^\infty(\R)$
                such that $\left.h\right|_{(-\infty,\frac{1}{2})}
                \equiv 1$ and
                $\left.h\right|_{(1,+\infty)}\equiv 0$, 
                $k\in(0,1)$, set
                \begin{equation*}
                        \chi(x,y,\xi)=
                        h\left(\frac{|x-\nabla_\xi\varphi(y,\xi)|}{k\norm{x}}\right).
                \end{equation*}
        }, let us assume that $\chi$ (is smooth and) 
        satisfies $S^{0,0,0}$ estimates on $\supp b$, 
        and set $\widetilde{b}=(1-\chi) b$. Then, defining
        \begin{equation}
                \label{eq:tF}
                \widetilde{F}(x,y)=\int_\Rn
                e^{i[\langle x,\xi\rangle - \varphi(y,\xi)]} \widetilde{b}(x,y,\xi)\,d\xi,
        \end{equation}
        it follows 
        that $\widetilde{F}\in\cS(\R^n\times\R^n)$, which implies
        that
               \begin{equation}
                                \label{eq:tFaQ}
                                \int_\Rn \left|\int_\Rn
      \widetilde{F}(x,y)\, a_Q(y)\,dy\right|\,dx \le C,
                \end{equation}
                with a constant $C$ independent of $a_Q$.
\end{prop}
\begin{proof} We will show that kernel
        $\widetilde{F}$ satisfies 
        \begin{equation}\label{EQ:F-Schwartz}
        \partial^\alpha_x\partial^\beta_y\widetilde{F}(x,y)
        \prec (\norm{x}\norm{y})^
        {-N},
        \end{equation} 
         for all $N\in\N$, $x,y\in\R^n$ and
        all multi-indices $\alpha,\beta$. 
        By the hypotheses on $b$ and
        $\varphi$, it is clear that it is 
        enough to prove the estimate only 
        for $\alpha=\beta=0$ and arbitrary order 
        in $x,y,\xi$ for $\widetilde{b}$.\\
        \noindent
        Indeed, $|x-\nabla_\xi\varphi(y,\xi)|\succ \norm{x}$ 
        on $\supp \widetilde{b}$,
        so that the operator $L_\xi$, acting on functions 
        $v=v(x,y,\xi)$ with respect to $\xi$ as
        \begin{equation*}
                (L_\xi v)(x,y,\xi)=
                \sum_{j=1}^n i \partial_{\xi_j}
                                                \left(
                                                        \frac{x_j-\partial_{\xi_j}\varphi(y,\xi)}
                                                               {|x-\nabla_\xi\varphi(y,\xi)|^2}
                                                        v(x,y,\xi)
                                                \right),
        \end{equation*}
        is well defined on $\supp\widetilde{b}$. 
        Moreover, on $\supp\widetilde{b}$,
        we have
        \begin{equation*}
                |x-\nabla_\xi\varphi(y,\xi)|\succ \norm{x}.
        \end{equation*}
        Then $|\nabla_\xi\va(y,\xi)|\leq |x-\nabla_\xi\va(y,\xi)|+
         |x|\prec |x-\nabla_\xi\va(y,\xi)|$, and it follows
         that we also have
          \begin{equation*}
     |x-\nabla_\xi\varphi(y,\xi)|\succ 
       \norm{\nabla_\xi\varphi(y,\xi)}\succ\norm{y}.
        \end{equation*}
        Now \eqref{EQ:F-Schwartz} follows by
        integrating by parts in \eqref{eq:tF}, observing that
        ${^tL_\xi} e^{i[\langle x,\xi\rangle - \varphi(y,\xi)]}
        =e^{i[\langle x,\xi\rangle - \varphi(y,\xi)]}$. 
        Then \eqref{eq:tFaQ} holds, since, for all $N\in\N$,
        we have
        \begin{align*}
                \int_\Rn \left|\int_\Rn\widetilde{F}(x,y)\, 
       a_Q(y)\,dy\right|dx &\le \int_\Rn \int_\Rn 
       |\widetilde{F}(x,y)|\,
       |a_Q(y)|\,dy\,dx
                \\
                &\le \widetilde{C} \int_\Rn\norm{x}^{-N}\,dx
           \int_\Rn|a_Q(y)|\,dy\le C\,|Q|\,|Q|^{-1}=C.
                \end{align*}
\end{proof}
Therefore, from now on we can then assume that
for some $k\in(0,1)$ we have
\begin{equation}\label{EQ:reduce-amp}
 \supp b \subseteq
D = \{(x,y,\xi)\in\R^n\times\R^n\times\R^n\,\colon\,
|x-\nabla_\xi\varphi(y,\xi)|\le k\norm{x}\}.
\end{equation} 
This implies that on $\supp b$ we have 
$\norm{x}\sim\norm{\nabla_\xi\varphi(y,\xi)}
\sim\norm{y}$ which in turn implies that 
$C_1 \norm{y}\le \norm{x}\le C_2 
\norm{y}$, $x,y\in\R^n$, for suitable constants $C_1,C_2>0$.
\begin{prop}\label{PROP:large-Q}
        Let $a_Q$ be an atom in $H^1(\R^n)$, 
        supported in a cube $Q\subset\R^n$ centred 
        at $y_0\in\R^n$ and with sidelength 
 $q\ge 1$ {\rm (}hence also $|Q|\ge 1${\rm )}. Then, estimate 
        \eqref{eq:estim} holds with a constant $C$ independent of
        $a_Q$.
\end{prop}
\begin{proof}
        Let us denote by $M_s$ the multiplication operator 
        $(M_s v)(x)=\norm{x}^s v(x)$.
        From composition formulae with pseudo-differential 
        operators (see \cite{RS07}) it follows that
        operator $M_\frac{n}{2}\cF$ is then 
        a Fourier integral operator 
        with amplitude bounded in $x$ and $y$,
        and of order $-\frac{n-1}{2}$ in $\xi$.
        Consequently, operator $M_\frac{n}{2}\cF$
        is bounded on  $L^2(\R^n)$
        by \cite{AF78} under assumption (I) and
        by \cite{RS06a} under assumptions (II) and (III).
        Applying H\"older's inequality and 
        denoting 
        $D_{q,y_0}=\{x\in\R^n\;|\; C_1\norm{y}\le\norm{x}
        \le C_2\norm{y}, y\in Q\}$, we get
        \begin{align*}
                \| \cF a_Q \|_{L^1(\R^n)} &= 
                   \int\limits_{\begin{subarray}{c}
                        \norm{x}\sim\norm{y}  \\
                           y\in Q 
           \end{subarray}} |\norm{x}^{-\frac{n}{2}} 
           (M_\frac{n}{2}\cF a_Q)(x)|dx
                \\
                &\le \left(\int_{D_{q,y_0}} \norm{x}^{-n}\,dx\right)^\frac{1}{2}
                       \| (M_\frac{n}{2}\cF)a_Q\|_{L^2(\R^n)}
                \\
                &\le \widetilde{C} \| a_Q \|_{L^2(\R^n)} 
                       \left[
                       \int_{D_{q,y_0}}
                            (1+|x|^2)^{-\frac{n}{2}}\,dx\right]^\frac{1}{2}
                \\
                &= \widetilde{C} | Q |^{-1} |Q|^\frac{1}{2} 
                     \left[
                       \int_{D_{q,y_0}}
                            (1+|x|^2)^{-\frac{n}{2}}\,dx\right]^\frac{1}{2}
                \\
                &= \widetilde{C} 
                     \left[ | Q |^{-1}
                       \int_{D_{q,y_0}}
      (1+|x|^2)^{-\frac{n}{2}}\,dx\right]^\frac{1}{2}\le C,
        \end{align*}
     where $C\ge0$ does not depend on $a_Q$. Indeed,
     let us prove the boundedness of the expression in
     the last line. Let us set
     $A=1+\dfrac{|C_1^2-1|^\frac{1}{2}}{C_1}$.          
     The required boundedness is a consequence of the following
     steps:
     \begin{itemize}
       \item choose $\psi\in\cC^\infty(\R)$ supported 
       in $(-\infty,2]$, taking values in $[0,1]$, and 
       such that $\psi(t)=1$ for $t\in(-\infty,1]$.
       Set $\chi(q,y_0)=\psi\left(\dfrac{|y_0|}
       {Aq\sqrt{n}}\right)$ and let
        \begin{align*} 
          I_1&=\chi(q,y_0) \, |Q|^{-1}\int_{D_{q,y_0}} 
          (1+|x|^2)^{-\frac{n}{2}}\,dx,
                                                \\
          I_2&=(1-\chi(q,y_0))\, |Q|^{-1} \int_{D_{q,y_0}}
          (1+|x|^2)^{-\frac{n}{2}}\,dx;
        \end{align*}
   
   \item on the support of $\chi(q,y_0)$ we have 
   $|y_0| \le 2Aq\sqrt{n}$, so, for $x\in D_{q,y_0}$,
      \begin{align*}
        |x| < \norm{x}\le C_2 \norm{y} &\le C_2
        \sqrt{(|y-y_0|+|y_0|)^2+1}
                                                        \\
       &\le C_2 \sqrt{\left(\dfrac{q\sqrt{n}}{2}+2Aq
       \sqrt{n}\right)^2+1}\le K q,
      \end{align*}                       
      where $K>0$ is independent of $q\ge 1$ and 
      $y_0\in\R^n$. Then, $D_{q,y_0}\subset B_{Kq}(0)$,
      where $B_{Kq}(0)$ is the ball centred at the origin with
      radius $Kq$,  and we have
          \begin{equation*}
            I_1 \le |Q|^{-1}\,|B_{Kq}(0)|\le K^n\, |B_1(0)|=B_1,
          \end{equation*}
           with $B_1>0$ independent of $q\ge 1$, $y_0\in\R^n$;
    \item on the support of $1-\chi(q,y_0)$ 
    we have $|y_0| \ge Aq\sqrt{n}> 1$ and, for $x\in D_{q,y_0}$,
    we have
                 \begin{align*}
      \sqrt{C_1^2|y|^2+C_1^2-1}\le |x| \le 
      \sqrt{C_2^2 |y|^2 + C_2^2-1}, \,y\in Q.
                 \end{align*}       
  Note also that, on the support of $1-\chi(q,y_0)$, 
  for $y\in Q$ we have $\dfrac{|y-y_0|}{|y_0|}\le 
    \dfrac{q\sqrt{n}}{2}
    \,\dfrac{1}{Aq\sqrt{n}}=\dfrac{1}{2A}<\dfrac{1}{2}$
    and
    $$|y|\ge|y_0|-|y-y_0|=|y_0|
    \left(1-\dfrac{|y-y_0|}{|y_0|}
     \right)\ge|y_0|\left(1-\dfrac{1}{2A}\right)>
     \dfrac{|y_0|}{2}>\dfrac{1}{2}.$$
   Hence we can estimate
   \begin{align*}
        C_1^2|y|^2+C_1^2-1&\ge |y_0|^2\left[C_1^2
        \left(1-\dfrac{|y-y_0|}{|y_0|}\right)^2+
        \dfrac{C_1^2-1}{|y_0|^2}\right]
          \\
       &\ge|y_0|^2\left[C_1^2\left(1-\dfrac{1}{2A}
       \right)^2-\dfrac{|C_1^2-1|}{|y_0|^2}\right]
           \\
       &\ge|y_0|^2\left[\dfrac{C_1^2(2A-1)^2}{4A^2}-
       \dfrac{|C_1^2-1|}{A^2q^2n}\right]
                                \\
       &\ge|y_0|^2\dfrac{C_1^2q^2n\left(1+
       \dfrac{2|C_1^2-1|^\frac{1}{2}}{C_1}\right)^2-
       4|C_1^2-1|}{4A^2q^2n}
                                \\
       &\ge|y_0|^2\dfrac{q^2n(C_1^2+
       4C_1|C_1^2-1|^\frac{1}{2})}{4A^2q^2n}>0,
       \end{align*}
 from which we get that
 $$r_1:=\min_{y\in Q}\sqrt{C_1^2|y|^2+C_1^2-1}
 \ge K_1|y_0|>0,$$     
  with $\dfrac{3C_1}{2}>K_1>0$ independent of $q\ge1$, 
  $y_0\in\R^n$. 
  Since $C_2\ge C_1$, on the support of $1-\chi(q,y_0)$ 
  we have $C_2^2 |y|^2 + C_2^2-1>0$, and
   \begin{align*}
    \sqrt{C_2^2 |y|^2 + C_2^2-1}&\le \sqrt{C_2^2 
    (|y_0|+|y-y_0|)^2 + C_2^2}
                                                        \\
   &\le C_2|y_0|\sqrt{\left(1+\dfrac{|y-y_0|}{|y_0|}\right)^2+
   \dfrac{1}{|y_0|^2}}
                                \\
    &\le 2C_2|y_0|,
   \end{align*}
  so that $\displaystyle r_1<r_2:=\max_{y\in Q}
  \sqrt{C_2^2 |y|^2 + C_2^2-1}\le K_2|y_0|$ with
  $K_2>K_1>0$ independent of $q\ge1$, $y_0\in\R^n$;
  we have then proved that, on the support of $1-\chi(q,y_0)$, 
  $D_{q,y_0}\subset\overline{B_{r_2}(0)\setminus B_{r_1}(0)}$,
  hence 
   \begin{align*}
   I_2 &\le (1-\chi(q,y_0))\, |Q|^{-1}\, |B_1(0)| 
   \int_{r_1}^{r_2}\frac{r^{n-1}}{(1+r^2)^{\frac{n}{2}}}\,dr
      \\
   &\le |B_1(0)| \int_{r_1}^{r_2}\frac{dr}{r}\le |B_1(0)| 
   \log\frac{K_2}{K_1}=B_2,
                        \end{align*}                    
   with $B_2>0$ independent of $q\ge1$, $y_0\in\R^n$.
                \end{itemize}
  The proof is complete.
\end{proof}

\begin{rem}\label{REM:large-Q}
Let operator $\cF$ be as in Theorem \ref{thm:main}
with $\mu$ satisfying \eqref{EQ:m-mu} but with
any $m\leq 0$. Let $R>0$. 
Let $a_Q$ be an atom in $H^1(\R^n)$, 
        supported in a cube $Q\subset\R^n$ centred 
        at $y_0\in\R^n$ and with sidelength $q$ such that
 $R\ge q\ge 1$. Then, estimate 
        \eqref{eq:estim} holds with a constant $C$ independent of
        such $a_Q$.
\end{rem}
This remark follows immediately 
from the proof of Proposition \ref{PROP:large-Q}
if we observe that 
the boundedness of $I_1$ is actually independent of
the order of $b$ in $x$, while the boundedness of
$I_2$ is a consequence of the fact that 
the volume of $D_{q,y_0}$ is bounded by a uniform constant
for all cubes $Q$ in Remark \ref{REM:large-Q}.

Of course, the argument in the proof of Proposition
\ref{PROP:large-Q}
still holds if the hypothesis 
$|Q|\ge1$ is replaced  by $|Q|\ge Q_0>0$, or, 
 equivalently, by 
$q\ge q_0 > 0$. In the next steps of the 
proof we can then assume that $a_Q$ is supported in a 
cube $Q$ with sidelength $q=2^{-j}$, $j\ge j_0$, 
where $j_0$ is chosen so large that $\dfrac{q}{2}\sqrt{n}<1$. 
In this way, $y\in Q\Rightarrow |y-y_0|\le \dfrac{q}{2}\sqrt{n}
\Rightarrow\norm{y}\sim\norm{y_0}$, $y_0$ centre of $Q$, 
so that we also have, on $\supp b$, that $\norm{x}\sim\norm{y_0}$.

We now define an ``exceptional set'' set $\cN_Q$, which covers 
\begin{equation}
        \label{eq:S}
        \S=\{x=\nabla_\xi\varphi(y,\xi)
        \mbox{ for some $y\in Q$, $\xi\in\R^n$}\},
\end{equation}
and use again $L^2$-boundedness results, together with 
H\"older and Hardy-Littlewood-Sobolev inequalities, to estimate
$\|\cF a_Q\|_{L^1}$ on that set.

\noindent
Choose unit vectors $\xi_k^\nu$, $\nu=1,\dots,N(k,y)$, $
k\ge j_0$, $y\in\R^n$, such that:
\begin{itemize}
        \item[-] $|\xi_k^\nu-\xi_k^{\nu^\prime}|\ge 
        C_02^{-\frac{k}{2}}\norm{y}^{-\frac{1}{2}}$, $\nu\not=\nu^\prime$,
                    for some fixed positive constant $C_0<1$;
        \item[-] the unit sphere ${\mathbb S}^{n-1}$ 
        is covered by the balls centred at $\xi_k^\nu$ 
        with radius 
                    $2^{-\frac{k}{2}}\norm{y}^{-\frac{1}{2}}$.
\end{itemize}
We have then $N(k,y)\approx 2^{k\frac{n-1}{2}}
\norm{y}^{\frac{n-1}{2}}$. For $y\in Q$ and a constant 
$M$ to be fixed later, define 
\begin{equation}
                                \label{eq:Ryknu}
                \cR^y_{k\nu}=   \left\{
            x\colon\! |\langle x - \nabla_\xi
            \varphi(y,\xi^\nu_k),\xi^\nu_k\rangle| \le M 2^{-k}
                                                \mbox{ and }
          |\Pi^\perp_{k\nu}(x - 
          \nabla_\xi\varphi(y,\xi^\nu_k))| \le M 2^{-\frac{k}{2}}\norm{y}^{\frac{1}{2}}
                                        \right\},
\end{equation}
where $\Pi^\perp_{k\nu}$ is the projection onto the plane 
orthogonal to $\xi^\nu_k$. Set $\cR^y_{k\nu}$ is then a $n$-rectangle
with $n-1$ sides of length $M 2^{-\frac{k}{2}} 
\norm{y}^{\frac{1}{2}}$ and one side of length $M 2^{-k}$. 
If $Q$ has sidelength $q=2^{-j}$,
$j\ge j_0$, we define
\begin{equation}
        \label{eq:NQ}
        \cN_Q = \bigcup_{y\in Q} \bigcup_{\nu=1}^{N(j,y)} \cR^y_{j\nu}.
\end{equation}
Since $|\cR^y_{j\nu}|\approx 2^{-j\frac{n+1}{2}}
\norm{y_0}^{\frac{n-1}{2}}$ for $y\in Q$, it follows that
\begin{equation}
                \label{eq:sizeNQ}
        |\cN_Q| 
   \le C 2^{j\frac{n-1}{2}}\norm{y_0}^{\frac{n-1}{2}} 
   2^{-j\frac{n+1}{2}} \norm{y_0}^{\frac{n-1}{2}}
                = C 2^{-j} \norm{y_0}^{n-1} = C  
                |Q|^\frac{1}{n} \norm{y_0}^{n-1},
\end{equation}
for some constant $C\ge 0$ independent of $j\ge j_0$, $y_0\in\R^n$.
\begin{lemma}
        \label{lemma:SNQ}
        If in \eqref{eq:Ryknu} we take 
        $\displaystyle M=\sup_{\begin{subarray}{c} 
        |\alpha|=2,3\\(y,\xi)\in\R^n\times\R^n\end{subarray}}
        \norm{y}^{-1}\norm{\xi}^{-1+|\alpha|}
        |\partial^\alpha_{\xi}\varphi(y,\xi)|$,
        the singular set $\Sigma$ defined in 
        \eqref{eq:S} is a subset of $\cN_Q$.
\end{lemma}
\begin{proof}
        Let us denote $\vers(\xi)=\frac{\xi}{|\xi|}$.
        Since, for all $\xi\in\R^n$, 
        $|\vers(\xi) - \xi_j^\nu|\le 2^{-\frac{j}{2}}
        \norm{y}^{-\frac{1}{2}}$ for some 
        $\nu=1, \dots, N(j,y)$, then,
        with $M$ chosen as above, 
        we have $\nabla_\xi\varphi(y,\xi)\in\cR^y_{j\nu}$. 
        Indeed, $\nabla_\xi\varphi(y,\xi)$ is
        homogeneous of order $0$ in $\xi$ and 
        $\Pi^\perp_{j\nu}$ is a projection, so that
        \begin{align*}
                |\Pi^\perp_{j\nu}(\nabla_\xi\varphi(y,\xi) - \nabla_\xi\varphi(y,\xi^\nu_j))|
                &\le |\nabla_\xi\varphi(y,\vers(\xi)) - \nabla_\xi\varphi(y,\xi^\nu_j)|
                \\
                &\le M \norm{y} |\vers(\xi)-\xi^\nu_j| \le M 2^{-\frac{j}{2}}\norm{y}^\frac{1}{2}.
        \end{align*}
        Moreover, again in view of the 
        homogeneity of the phase function, if we set
        $h^\nu_j(y,\xi)=\langle\nabla_\xi
        \varphi(y,\xi),\xi_j^\nu\rangle-
        \langle\nabla_\xi\varphi(y,\xi_j^\nu),\xi_j^\nu\rangle
        =\langle\nabla_\xi\varphi(y,\xi),
        \xi^\nu_j\rangle-\varphi(y,\xi_j^\nu)$, we have 
        $h^\nu_j(y,\xi^\nu_j) = 0$ and
        $\nabla_\xi h^\nu_j(y,\xi) = 
                \langle \varphi^{\prime\prime}_{\xi\xi}(y,\xi), 
                \xi^\nu_j \rangle$. Therefore, we get
        $ \nabla_\xi h^\nu_j(y,\xi^\nu_j)=0$
        by Euler's formula. Writing the Taylor expansion
        of $h^\nu_j(y,\xi)$
        with respect to $\xi$ at $\xi_j^\nu$, we obtain
        \begin{equation*}
                |h^\nu_j(y,\xi)|\le M\norm{y}
                |\vers(\xi)-\xi^\nu_j|^2\le M 2^{-j},
        \end{equation*}
        as desired.
\end{proof}

\begin{prop}\label{PROP:NQ}
   $\displaystyle\|\cF a_Q\|_{L^1(\cN_Q)}\le C$ 
   with $C$ independent of $a_Q$.
\end{prop}
\begin{proof}
 First we observe that operator 
 $M_\frac{n}{2}\cF(1-\Delta)^\frac{n-1}{4}$ is a
 Fourier integral operator with the same phase and same
 properties of the amplitude as those of $\cF$ in view
 of the global calculus in \cite{RS07}.
 Consequently, operator
 $M_\frac{n}{2}\cF(1-\Delta)^\frac{n-1}{4}$
 is bounded on $L^2(\Rn)$ in view of
 the $L^2$-boundedness theorems in \cite{AF78} under
 assumption (I) and in \cite{RS06a} under assumptions 
 (II) and (III).
                 Writing $p_n=\dfrac{2n}{2n-1}$ and
                recalling \eqref{eq:sizeNQ}, 
                we have
               \begin{align*}
       \|\cF a_Q\|_{L^1(\cN_Q)} &= \left\|M_{-\frac{n}{2}} 
       \left[M_\frac{n}{2}\cF(1-\Delta)^\frac{n-1}{4}\right]
       (1-\Delta)^{-\frac{n-1}{4}} 
          a_Q\right\|_{L^1(\cN_Q)}   \\
            &\le \left(\int\limits_{\begin{subarray}{c}
            \cN_Q\\\norm{x}\sim\norm{y}\end{subarray}} 
            \norm{x}^{-n}\,dx\right)^\frac{1}{2}
             \left\|\left[M_\frac{n}{2}\cF
             (1-\Delta)^\frac{n-1}{4}\right] 
      \left[(1-\Delta)^{-\frac{n-1}{4}} a_Q\right]
      \right\|_{L^2(\R^n)}
                                                \\
        &\le C_1\left(\norm{y_0}^{-n}\,|Q|^\frac{1}{n}\,
        \norm{y_0}^{n-1}\right)^\frac{1}{2}
        \left\|(1-\Delta)^{-\frac{n-1}{4}} a_Q
        \right\|_{L^2(\R^n)}
                                                \\
        &\le C_2\,|Q|^\frac{1}{2n}\left\|a_Q
        \right\|_{L^{p_n}(\R^n)}
        \le C\, |Q|^\frac{1}{2n}\,|Q|^{-\frac{1}{2n}} = C,
                \end{align*}
    with a constant $C$ independent of $a_Q$, 
    in view of the Hardy-Littlewood-Sobolev inequality
                \begin{equation*}
                                \left\|(1-\Delta)^{-\frac{n-1}{4}} a_Q\right\|_{L^2(\R^n)}\le\widetilde{C}\left\|a_Q\right\|_{L^{p_n}(\R^n)},
                \end{equation*}
                and since, obviously, $\|a_Q\|_{L^{p_n}(\R^n)}\le |Q|^{\frac{1}{p_n}-1}=|Q|^{-\frac{1}{2n}}$.
\end{proof}

\noindent
We will now prove the estimate
\begin{equation}
        \label{eq:offNQ}
        \|\cF a_Q\|_{L^1(\R^n\setminus\cN_Q)}\le C
\end{equation}
off the exceptional set. We first introduce a 
dyadic decomposition, choosing 
function $\theta\in\cC^\infty(\R)$ such that
$\supp\theta\subset\left(\dfrac{1}{4},4\right)$ and 
such that for all $s>0$ 
we have $\sum_{k\in\Z}\theta(2^{-k}s)=1$. We now set
\begin{equation}
        \label{eq:defFk}
        F_k(x,y)=\int_\Rn e^{i[\langle x,\xi\rangle-
        \varphi(y,\xi)]} b(x,y,\xi)\,\theta_k(\xi)\,d\xi,
\end{equation}
where $\theta_k(\xi)=\theta(2^{-k}|\xi|)$. 
We can assume without loss of generality 
that $b(x,y,\xi)=0$ for $|\xi|<8$. 
Defining
$\displaystyle\theta_0=1-\sum_{k>0}\theta_k$, 
we have $F=\displaystyle\sum_{k\ge1}F_k$. 
Estimate \eqref{eq:offNQ} is then a consequence
of the following proposition, where we recall that
$j$ was introduced in a way that $2^{-j}$ is a sidelength
of $Q$.
\begin{prop}\label{PROP:NQc}
        For all $y,y^\prime\in Q$, $j,k\in\N$, $j\ge j_0$,
        we have
        \begin{align}
                                \label{eq:estFk-a}
                \int_{\R^n\setminus\cN_Q}|F_k(x,y)|\,dx 
                \prec 2^{j-k}&\mbox{ if $k>j$},
                \\
                                \label{eq:estFk-b}
                \int_{\R^n}|F_k(x,y)-F_k(x,y^\prime)|\,dx 
                \prec 2^{k-j}&\mbox{ if $k\le j$}.
        \end{align}
\end{prop}
\begin{proof}
        For each $k\in\N$, let $\{\chi^\nu_{k}\}$, 
        $\nu=1,\dots,N(y,k)$, be a homogeneous partition 
        of unity associated with the covering
        of the unit sphere with the balls 
        $B(\xi_k^\nu,c_02^{-\frac{k}{2}}\norm{y}^{-\frac{1}{2}})$, 
        as introduced above. Explicitly,
        we choose $\cC^\infty$ functions 
        $\chi^\nu_k=\chi^\nu_k(y,\xi)$, homogeneous 
        in $\xi$ of degree $0$, such that,
        for all $y\in\R^n$, we have
        \begin{itemize}
                \item[-] $\chi^\nu_k(y,\vers(\xi))\equiv 1$ 
                for $\vers(\xi)$ in a neighbourhood of 
                $\xi^\nu_k$ in ${\mathbb S}^{n-1}$;
                \item[-] $\chi^\nu_k(y,\xi)=0$ if 
                $|\vers(\xi)-\xi^\nu_k|\ge c_02^{-\frac{k}{2}}
                \norm{y}^{-\frac{1}{2}}$; 
                \item[-] $\sum_\nu\chi^\nu_k=1$;
                \item[-] $|\partial^\gamma\chi^\nu_k(y,\xi)|
                \prec |\xi|^{-|\gamma|}(2^k\norm{y})^
                {\frac{|\gamma|}{2}}$ for all 
                multi-indices $\gamma\in\Z_+^n$.
        \end{itemize}
        We now define
        \begin{equation*}
                F_k^\nu(x,y)=\int_\Rn 
                e^{i[\langle x,\xi\rangle-\varphi(y,\xi)]} 
                b^\nu_k(x,y,\xi)\,d\xi,
        \end{equation*}
        where $b^\nu_k(x,y,\xi)=b(x,y,\xi)\,\theta_k(\xi)\,
        \chi^\nu_k(y,\xi)$. Set also
        \begin{equation*}
                r^\nu_k(y,\xi)=\varphi(y,\xi)-\langle\nabla_\xi
                \varphi(y,\xi^\nu_k),\xi\rangle
                \Rightarrow \nabla_\xi r^\nu_k(y,\xi)=\nabla_\xi
                \varphi(y,\xi)-\nabla_\xi\varphi(y,\xi^\nu_k),
        \end{equation*}
        and $D^\nu_k=\langle\nabla_\xi,\xi^\nu_k\rangle$, 
        $\nu=1,\dots,N(k,y)$. Clearly, by definition of 
        $r^\nu_k$ and homogeneity of $\varphi$, we have
        $r^\nu_k(y,\xi^\nu_k)=0$ and
        $\nabla_\xi r^\nu_k(y,\xi^\nu_k)=0$. Since, 
        again by homogeneity,
        \begin{align*}
                (D^\nu_k r^\nu_k)(y,\xi)&= D^\nu_k\varphi(y,\xi)
                -\varphi(y,\xi^\nu_k)
                \Rightarrow (D^\nu_k r^\nu_k)(y,\xi^\nu_k)=0,
                                \\
            (\nabla_\xi D^\nu_k r^\nu_k)(y,\xi)&= D^\nu_k
            \nabla_\xi\varphi(y,\xi)
            \Rightarrow (\nabla_\xi D^\nu_k r^\nu_k)
            (y,\xi^\nu_k)=0,
        \end{align*}
        by induction we also see that, for all $N\in\N$, we have
                \begin{equation}
                        \label{eq:rnu-rad-a}
                        [(D^\nu_k)^N r^\nu_k](y,\xi^\nu_k)=0, 
                        \hspace{0.5cm} [\nabla_\xi(D^\nu_k)^N 
                        r^\nu_k](y,\xi^\nu_k)=0.
                \end{equation}
                Writing the Taylor expansion in $\xi$ of $r^\nu_k$ 
                centred in $\xi^\nu_k$, \eqref{eq:rnu-rad-a} 
                implies that, for all
                $N\in\N$, on $\supp(b^\nu_k)$ we have
                \begin{equation}
                        \label{eq:rnu-rad-b}
                        [(D^\nu_k)^N r^\nu_k](y,\xi)\prec 
                        |\xi|^{1-N}\norm{y}|\vers(\xi)-\xi^\nu_k|^2
                        \prec 2^{k(1-N)}2^{-k}=2^{-kN}.
                \end{equation}
                On the other hand, for the ``transversal 
                derivatives" with $|\gamma|\ge 1$ we have, 
                on $\supp(b^\nu_k)$,
                \begin{equation}
                        \label{eq:rnu-tr}
                        D^\gamma_\xi r^\nu_k(y,\xi)\prec 
                        |\xi|^{1-|\gamma|}\norm{y}\prec 
                        2^{-k(|\gamma|-1)}\norm{y}\prec
                        2^{-k\frac{|\gamma|}{2}}\norm{y}.
                \end{equation}
                Indeed, first we recall that
                on $\supp(b^\nu_k)$, $|\xi|$ is 
                equivalent to $2^{k}$.
                For $|\gamma|\ge 2$, we then have
                $|\xi|^{1-|\gamma|}
                \prec 2^{k(1-|\gamma|)}\le 2^{-k\frac{|\gamma|}{2}}$ 
                and hence also \eqref{eq:rnu-tr}. 
                For $|\gamma|=1$, the first derivatives 
                are actually bounded by $2^{-\frac{k}{2}}
                \norm{y}^{\frac{1}{2}}$,
                since by $\nabla_\xi r^\nu_k(y,\xi^\nu_k)=0$ 
                and Taylor expansion we have
                \begin{equation*}
                        (\partial_{\xi_j}r^\nu_k)(y,\xi)
                        \prec \norm{y}|\xi-\xi^\nu_k|\prec 
                        2^{-\frac{k}{2}}\norm{y}^\frac{1}{2}.
                \end{equation*}
      Consequently, one can readily check that on $\supp(b^\nu_k)$,
            we have estimate
                                \begin{equation}
                                        \label{eq:exprnu}
        D^\gamma_\xi e^{ir^\nu_k(y,\xi)}\prec 
        2^{-k\frac{|\gamma|}{2}}\norm{y}^\frac{|\gamma|}{2}.
                                \end{equation}
     Performing a rotation\footnote{Note that all 
     the symbol estimates for $\theta_k$,
     $\chi^\nu_k$, $r^\nu_k$, $\varphi$, and $b$ hold unchanged
     for fixed $y$, since all the entries of $C$ are bounded, 
     in view of $A\in O(n)$.} $\xi=C\widetilde{\xi}$,
     we can simplify notation and assume 
     $\xi^\nu_k=(1,0,\dots,0)$, $\Pi^\nu_k(\xi)=(0,\xi^\prime)$. 
     Rewriting $F^\nu_k(x,y)$ as
                                \begin{equation}
                                        \label{eq:Fnuk}
                                        F^\nu_k(x,y)=\int_\Rn 
                 e^{i\langle x-\nabla_\xi\varphi(y,\xi),
                 \xi^\nu_k \rangle}\widetilde{b}^\nu_k(x,y,\xi)\,
                 d\xi,
                                \end{equation}
                                where $\widetilde{b}^\nu_k(x,y,\xi)=e^{ir^\nu_k(y,\xi)}
                                b^\nu_k(x,y,\xi)$, we observe that the 
                                derivatives in the $\xi_1$ (``radial'') direction
                                of $\chi^\nu_k$ vanish identically, 
                                so that, defining the 
                                selfadjoint operator $L^\nu_k$ as
                                \begin{equation*}
                                        L^\nu_k=\left(I-2^{2k}
           \frac{\partial^2}{\partial \xi_1^2}\right)
           \left(I-2^{k}\norm{y}^{-1}\langle\nabla_{\xi^\prime},
           \nabla_{\xi^\prime}\rangle\right),
                                \end{equation*}
           \eqref{eq:rnu-rad-b}, \eqref{eq:exprnu}, 
           the properties of $\chi_k^\nu$, the definition 
           of $\theta_k$ and the hypoteses on $\varphi$ and $b$
                                imply, for all $N\in\N$, that we have
                                \begin{equation}
                                        \label{eq:estbnuk}
           [(L^\nu_k)^N\widetilde{b}^\nu_k](x,y,\xi)\prec 
           2^{-k\frac{n-1}{2}}\norm{y}^{-\frac{n}{2}}.
                                \end{equation}
           Repeated integrations by parts allow to write
                                \begin{equation*}
                                        F^\nu_k(x,y)=
                                        H^\nu_{k,N}(x,y)
                                     \int_\Rn e^{i\langle x-\nabla_\xi
\varphi(y,\xi^\nu_k),\xi \rangle}
[(L^\nu_k)^N\widetilde{b}^\nu_k](x,y,\xi)\,d\xi,
                                \end{equation*}
                                with
         \begin{equation*}
           H^\nu_{k,N}(x,y)=\left(1+|2^k(x-\nabla_\xi
           \varphi(y,\xi^\nu_k))_1|^2\right)^{-N}
           \left(1+|2^\frac{k}{2}\norm{y}^{-\frac{1}{2}}
           (x-\nabla_\xi\varphi(y,\xi^\nu_k))^\prime|^2\right)^{-N}.
         \end{equation*}
   Since $\textrm{vol}_\xi(\supp(\widetilde{b}^\nu_k))
   \prec 2^k \, (2^k\cdot2^{-\frac{k}{2}}
   \norm{y}^{-\frac{1}{2}})^{n-1}=2^{k\frac{n+1}{2}}
   \norm{y}^{-\frac{n-1}{2}}$, by
                                \eqref{eq:estbnuk} it follows that
                                \begin{equation}
                                        \label{eq:estFnuk}
|F^\nu_k(x,y)|\prec H^\nu_{k,N}(x,y) \, 2^{k} 
\norm{y}^{-n+\frac{1}{2}}.
                                \end{equation}
 In $\R^n\setminus\cN_Q$, we must have either 
 $|2^k(x-\nabla_\xi\varphi(y,\xi^\nu_k))_1|\succ 2^{k-j}$ or
 $|2^\frac{k}{2}\norm{y}^{-\frac{1}{2}}(x-\nabla_\xi
 \varphi(y,\xi^\nu_k))^\prime|\succ 2^{\frac{k-j}{2}}$. 
 Since, obviously, $H^\nu_{k,N}=H^\nu_{k,N-N^\prime}\cdot 
 H^\nu_{k,N^\prime}$ for any $N,N^\prime\in\N$ such
 that $N>N^\prime$, then, for any $k>j$, we can estimate 
  \begin{equation}
  \label{eq:intHnuk}
  \int_\Rn H^\nu_{k,N}(x,y)\,dx \le C_{N-N^\prime}\,2^{-k}\,
  2^{-k\frac{n-1}{2}}\norm{y}^{\frac{n-1}{2}}\,2^{-N^\prime(k-j)},
  \end{equation}
  which implies, together with \eqref{eq:estFnuk}, that 
     \begin{equation}
       \label{eq:estintFnuk}
       \int_\Rn |F^\nu_k(x,y)|\,dx\prec 2^{j-k} \,
       2^{-k\frac{n-1}{2}}\norm{y}^{-\frac{n}{2}}.
     \end{equation}
     
     Now \eqref{eq:estFk-a} follows from \eqref{eq:estintFnuk}, 
     by summing over $\nu=1,\dots,N(y,k)$.
      Owing to
      \begin{align*}
       \int_\Rn |F_k(x,y)&-F_k(x,y^\prime)|\,dx\le
       \sum_\nu\int_\Rn|F^\nu_k(x,y)-F^\nu_k(x,y^\prime)|\,dx
         \\
       &\le|y-y^\prime|\sum_\nu\int_\Rn\sup_{y\in Q}
       |\nabla_yF^\nu_k(x,y)|\,dx\prec 2^{-j}
       \sum_\nu\int_\Rn\sup_{y\in Q}|\nabla_yF^\nu_k(x,y)|\,dx,
      \end{align*}
     estimate \eqref{eq:estFk-b} would follow from
       \begin{equation}
         \label{eq:estdyFk}
         \int_\Rn\sup_{y\in Q}|\nabla_yF^\nu_k(x,y)|\,dx
         \prec 2^k\cdot 2^{-k\frac{n-1}{2}}
         \norm{y_0}^{-\frac{n}{2}}.
       \end{equation}
      Now, \eqref{eq:estdyFk} indeed holds true, since 
      $\nabla_yF^\nu_k(x,y)$ can be written in the 
      form \eqref{eq:Fnuk} with
      $\widetilde{a}^\nu_k(x,y,\xi)=
      \nabla_y\widetilde{b}^\nu_k(x,y,\xi)-
      i\widetilde{b}^\nu_k(x,y,\xi)\cdot\nabla_y
      \varphi(y,\xi)$ in place of $\widetilde{b}^\nu_k(x,y,\xi)$,
      and $\widetilde{a}^\nu_k(x,y,\xi)$ 
      has the same properties of
       $\widetilde{b}^\nu_k(x,y,\xi)$ with order in 
       $\xi$ increased by one unit. It is then possible to 
      repeat the same argument used in the proof of 
      \eqref{eq:estintFnuk}, and to sum over 
      $\nu=1,\dots,N(y,k)$, recalling that
      $\norm{y}\sim\norm{y_0}$ for $y\in Q$.
\end{proof}

\noindent
\textit{Conclusion of the proof of \eqref{eq:offNQ}}:
by properties (1), (2) and (3) of $a_Q$ and Proposition \ref{PROP:NQc}, denoting by $\cF_k$ the operator with kernel $F_k$ defined in \eqref{eq:defFk}, we have
\begin{align*}
        \| \cF a_Q \|_{L^1(\R^n\setminus\cN_Q)} &\le \sum_{k\ge 0}\| \cF_k a_Q \|_{L^1(\R^n\setminus\cN_Q)}
        \\
        &\le\sum_{0\le k\le j}\int_{\R^n}\left|\int_Q [F_k(x,y)-F_k(x,y^\prime)]\, a_Q(y)\,dy\right|dx
        \\
        &+\sum_{k> j}\int_{\R^n\setminus\cN_Q}\left|\int_Q F_k(x,y)\, a_Q(y)\,dy\right|dx
        \\
        &\le\sum_{0\le k\le j}\int_Q\left[\int_{\R^n} | F_k(x,y)-F_k(x,y^\prime)| \,dx\right] |a_Q(y)|\,dy
        \\
        &+\sum_{k> j}\int_Q\left[\int_{\R^n\setminus\cN_Q} | F_k(x,y) |\,dx \right] |a_Q(y)|\,dy
        \\
        &\le \frac{C}{3}\left(\sum_{0\le k\le j} 2^{k-j}+\sum_{k> j} 2^{j-k}\right) \le C,
\end{align*}
with $C$ independent of $a_Q$, as claimed.

\begin{rem}\label{REM:small-Q}
 We note that statements of Propositions \ref{PROP:NQ}
 and \ref{PROP:NQc} remain true if operator $\cF$ satisfies
 assumptions of Theorem \ref{thm:submain} only with
 $m\leq -(n-1)/2$. 
\end{rem}

\section{Proof of Theorem \ref{thm:main2}}
\label{sec:sketch}

A preliminary result to be proven is the following

\begin{prop}[$L^p(\Rn)$-boundedness of localised 
Fourier integral operators]
\label{prop:dyadicFIO} 
        Assume the hypotheses in Theorem 
        \ref{thm:main2} and let
        $\tp\in\cC^\infty_0(\Rn)$ be supported in 
        the shell $2^{-2}\leq|x|\leq 2^2$. Then
        we have, for $k\geq1$,
        \[
                \|\tp(2^{-k}x)Af\|_{L^p}\leq C\|f\|_{L^p},
        \]
        where the constant $C$ depends only on $\tp$, on
        upper bounds for a finite number of the 
        constants in the estimates satisfied by
        $a$ and $\varphi$, and on the lower bound $\delta$ for the
        determinant of the mixed Hessian of $\varphi$.
\end{prop}

\begin{proof}
        We can write
        \[
                \tp(2^{-k}x)A=U_{2^{-k}} A'_k U_{2^{k}},
        \]
        where $U_\lambda f(x)=f(\lambda x)$, $\lambda\not=0$, is the dilation operator and
        \[
                A'_k f(x)=\int_\Rn e^{i\varphi(2^k x,2^{-k}\xi)}
        \tp(x) a(2^k x,2^{-k}\xi)\widehat{f}(\xi)\,d\xi.
    \]
   Hence it suffices to prove the desired conclusion with $A'_k$ in place
        of $\tp(2^{-k}x)A$. It follows from the estimates satisfied by $\varphi$ and the
        fact that $|x|\sim1$ on the support of $\tp$ that, there,
        \[
                |\partial^\alpha_x\partial^\beta_\xi(\varphi(2^k x,2^{-k}\xi))|\leq
                M_{\alpha,\beta}|\xi|^{1-|\beta|},
        \]
        (in fact, $\langle 2^{k}x\rangle\sim 2^k$ on the support of $\tp$).
        Moreover, we immediately have
        \begin{equation}
                \left|{\rm det}\,
                        \left(\frac{\partial^2(\varphi(2^k x,2^{-k}\xi))}{\partial \xi_j\partial x_l}\right)\right|>\delta>0.
        \end{equation}

        \noindent
        Similarly, one sees that\footnote{Precisely, to verify this last estimate, distinguish the case
        $|\xi|\leq 2^k$ (which implies $\langle 2^{-k}\xi\rangle\sim1$, 
        $\langle\xi\rangle\prec2^k \Rightarrow 
        1\prec 2^{k(-m_p+|\beta|)}\langle\xi\rangle^{m_p-|\beta|}$) and the case $|\xi|\geq 2^k$
        (which implies $\langle 2^{-k}\xi\rangle\sim|2^{-k}\xi|$, $\langle\xi\rangle\sim|\xi|$).
        }, on the support of $\tp$, we have
        \begin{align*}
                |\partial^\alpha_x\partial^\beta_\xi(a(2^k x,2^{-k}\xi))|&=
                2^{k(|\alpha|-|\beta|)}|(\partial^\alpha_x\partial^\beta_\xi a)(2^k x,2^{-k}\xi)|
                \\
                &\leq
                C_{\alpha,\beta} 2^{k(|\alpha|-|\beta|)}
                \norm{2^k x}^{m-|\alpha|}\norm{2^{-k}\xi}^{\mu-|\beta|}
                \\
                &\leq
                C_{\alpha,\beta} 2^{k(|\alpha|-|\beta|)}
                \norm{2^k x}^{m_p-|\alpha|}\norm{2^{-k}\xi}^{m_p-|\beta|}
                \\
                &\leq
                C_{\alpha,\beta} 2^{k(|\alpha|-|\beta|+m_p-|\alpha|-m_p+|\beta|)}
                \norm{\xi}^{m_p-|\beta|}
                \\
                &=C_{\alpha,\beta} \norm{\xi}^{m_p-|\beta|},
        \end{align*}
        where we have set $m_p=-(n-1)\left|\dfrac{1}{p}-\dfrac{1}{2}\right|\ge m,\mu$. 

        \noindent
        We have then showed that the operators $A'_k$ satisfy the assumptions of Seeger-Sogge-Stein's 
        Theorem, uniformly with respect to $k\in\N$: an application of that theorem concludes the
        proof\footnote{Indeed, it suffices to observe that the amplitudes of the $A^\prime_k$, $k\in\N$, are compactly    supported and all the other requirements of the Seeger-Sogge-Stein's Theorem are fulfilled; moreover, the   constant in the boundedness estimate of the aformentioned Theorem depends only on upper 
        bounds for a finite number of the constants in the estimates satisfied by
        the phase and amplitude functions, and a lower bound for the mixed Hessian of the phase.}.
\end{proof}

We then make use of a Littlewood--Paley partition 
of unity $\{ \psi_k \}$, $k\in\Z_+$, such that
$\psi_0\in C_0^\infty(\R^n)$, $\psi_k(x)=\psi(2^{-k}x)$, $k\ge 1$, $\supp \psi 
\subset\{x\in\Rn\colon 2^{-1}\le|x|\le2\}$,
and write the operator $A$ of \eqref{EQ:fio2} as
\begin{equation}
        \label{eq:dyadic}
        A=\psi_0 A+\sum_{k=1}^\infty\psi_k A.
\end{equation}
The operator $\psi_0 A$ is $L^p$-bounded by the 
Seeger-Sogge-Stein's theorem \cite{SSS91},
so we only treat the second term in 
\eqref{eq:dyadic}, namely, the sum over $k\geq1$, writing
   \[
        \sum_{k=1}^\infty\psi_k
        A= \sum_{k=1}^\infty\sum_{k'=0}^\infty\psi_{k}
        A\psi_{k'}.
  \]
        The functions $\psi_k$, $k\ge 1$, can be 
        interpreted as SG pseudo-differential operators,
   so that it is possible to use the composition 
   formulae of a SG Fourier integral operator 
   with a SG pseudo-differential
   operator, see \cite{Co99} or
   \cite{RS06b, RS07}. Splitting the asymptotic 
   expansion of the amplitude of the composed operator
   into the sum of the terms from order $(m,\mu)$ 
   to order $(m-3,\mu-3)$ and of the corresponding remainder,
   we write
   \begin{equation}\label{eq:pl}
        \psi_k A\psi_{k'}=A_{k,k'}+2^{-k-k'}R_{k,k'}.
   \end{equation}
Actually, we can compose the operators in \eqref{eq:pl} on the left
    with the multiplication by $\tp_k(x):=\tp(2^{-k}x)$,
     and on the right with the multiplication by $\tp_{k'}(x)$, 
     for a suitable cut off
    $\tp$, so that $\tp_k\psi_k=\psi_k$. 
    This does not affect the left-hand side and we find
    \[
        \psi_k A\psi_{k'}=\tp_kA_{k,k'}\tp_{k'}
        +2^{-k-k'}\tp_kR_{k,k'}\tp_{k'},
   \]
   with Fourier integral operators 
   $A_{k,k'}$ and $R_{k,k'}$, with amplitudes in 
   $S^{m,\mu}$ and  
   in $S^{m,\mu-2}$, respectively 
   ({\it uniformly with respect to} $k,k'$). 
   Note also that, in view of the properties of the
   Littlewood-Paley partition of unity and the formula 
   for the asymptotic expansion of the amplitude of the 
   composition of a pseudo-differential operator and a 
   Fourier integral operator,
    $|k-k'|>N$ implies $A_{k,k'}\equiv 0$, for some fixed $N>0$.
   Proposition \ref{prop:dyadicFIO} applied with 
   $A_{k,k'}$ in place of $A$
         and $\tp_{k^\prime}f$ in place of $f$, 
         together with the properties of
         the dyadic decomposition $\{\psi_k \}$, 
         $k\in \Z_+$, gives the desired estimate for the operator
         $\displaystyle \sum_{k=1}^\infty\sum_{k'=0}^\infty 
         \tp_kA_{k,k'}\tp_{k'}$:
         \begin{align*}
                \Big\|\sum_{k=1}^\infty&\sum_{k'\geq0,|k'-k|\leq N} \tp_k
                A_{k,k'}\tp_{k'}f\Big\|_{L^p}^p\prec
                \sum_{k=1}^\infty\Big\|\sum_{k'\geq0, |k'-k|\leq N} \tp_k
                A_{k,k'}\tp_{k'}f\Big\|_{L^p}^p
        \\
                        &\prec
                \sum_{k=1}^\infty\sum_{k'\geq0,|k'-k|\leq N}\|
                \tp_k
                A_{k,k'}\tp_{k'}f\|_{L^p}^p
        \\
                &\prec\sum_{k=1}^\infty\sum_{k'\geq0,|k'-k|\leq N}\|\tp_{k'}f\|_{L^p}^p
                        \le(2N+1)\sum_{k^\prime=0}^\infty\|\tp_{k'}f\|_{L^p}^p\prec\|f\|_{L^p}^p,
        \end{align*}
         where we used $\displaystyle\sum_{k^\prime=0}^\infty\|\tp_{k^\prime} f\|_{L^p}^p\prec \|f\|_{L^p}^p$,
         $\displaystyle\Big\|\sum_{k=1}^\infty \tp_k u_k\Big\|_{L^p}^p\prec\sum_{k=1}^\infty \|\tp_k u_k\|_{L^p}^p$,
         which hold for arbitrary $f,u_k\in L^p(\R^n)$, $k\ge 1$. A similar argument allows to estimate
        \[
                \|\sum_{k=1}^\infty\sum_{k'=0}^\infty 2^{-k-k'}\tp_{k}R_{k,k'}\tp_{k'}f\|_{L^p}
                \leq\sum_{k=1}^\infty\sum_{k'=0}^\infty2^{-k-k'}\|\tp_{k}R_{k,k'}\tp_{k'}f\|_{L^p}.
        \]
        Indeed, again by Proposition \ref{prop:dyadicFIO} 
        applied with $R_{k,k'}$ in place of $A$,
        and $\tp_{k^\prime}f$ in place of $f$, we see 
        that the right hand side is
        \[
                \prec\sum_{k=1}^\infty\sum_{k'=0}^\infty 2^{-k-k'}\|\tp_{k'}f\|_{L^p}=
                \sum_{k'=0}^\infty 2^{-k'}\|\tp_{k'}f\|_{L^p},
        \]
        and, by an application of H\"older's inequality, 
        the last expression is dominated by
        \[
                \prec\left(\sum_{k'=0}^\infty
                \|\tp_{k'}f\|^p_{L^p}\right)^{1/p}\prec \|f\|_{L^p}.
        \]

\section{Acknowledgements}
The authors would like to thank Fabio Nicola and 
Luigi Rodino for fruitful conversations and comments.





\end{document}